\documentclass[a4 paper,10pt]{article}
\usepackage{amsthm, a4wide}
\usepackage{amsfonts}
\usepackage{amsmath}
\usepackage{amssymb}
\usepackage{verbatim}

\usepackage{enumerate}
\usepackage[T1]{fontenc}
\usepackage[utf8]{inputenc}
\usepackage{mathtools}
\usepackage{dsfont} 
\usepackage{bbm}
\usepackage{tcolorbox}

\usepackage{colonequals}

\usepackage{xcolor}

\let\ol=\overline

\numberwithin{equation}{section}
\newtheorem{thm}{Theorem}[section]
\newtheorem{lem}[thm]{Lemma}
\newtheorem{cor}[thm]{Corollary}

\theoremstyle{definition}

\theoremstyle{remark}
\newtheorem{rem}[thm]{Remark}

\newcommand{\un}{\mathbf{1} }

\newcommand{\bP}{\mathbb{P}}
\newcommand{\bE}{\mathbb{E}}
\newcommand{\bS}{\mathbb{S}}

\newcommand{\R}{{\mathbb R}}
\newcommand{\N}{{\mathbb N}}
\newcommand{\Z}{\mathbb Z}
\newcommand{\T}{{\mathbb T}}

\newcommand{\Xc}{{\mathcal X}}

\newcommand{\Ac}{{\mathcal A}}

\newcommand{\bd}{{\rm \mathbf{  d}}}
\let\d=\delta
\let\e=\varepsilon 
\let\l=\lambda
\let\s=\sigma
\let\r=\rho 

\let\ol=\overline
\let \a =\alpha

\let \vf = \varphi

\newcommand{\mL}{\mathcal{L}}

\newcommand{\mA}{\mathcal{A}} 
\newcommand{\mB}{\mathcal{B}} 
\newcommand{\mF}{\mathcal{F}}
\newcommand{\mX}{\mathcal{X}}

\newcommand{\norma}[1]{\left \|#1\right \|}

\newcommand{\tr}{ \operatorname{tr}\, }

\def\bel{\begin{equation}\label}
\def\eeq{\end{equation}}

\def\bem*{\begin{multline*}}
\def\eem*{\end{multline*}}

\allowdisplaybreaks[4]
\begin{document}

\title{Ergodic Mean Field Games with H\"ormander diffusions}


\author{Federica Dragoni
\thanks{Cardiff School of Mathematics, Cardiff University, Cardiff, UK, e-mail: DragoniF@cardiff.ac.uk.} \and
 Ermal Feleqi 
\thanks{Cardiff School of Mathematics, Cardiff University, Cardiff, UK, e-mail:  FeleqiE@cardiff.ac.uk} 
}

\maketitle


\begin{abstract}
 We prove  existence of solutions for a class of systems of subelliptic PDEs 
 arising from Mean Field Game systems with H\"ormander diffusion.
These results are motivated by the feedback synthesis Mean Field Game solutions and the Nash equilibria of a large class of $N$-player differential games. 
 \end{abstract}




\section{Introduction} 

In this paper we consider
  a class of systems of degenerate elliptic PDEs of H\"ormander type arising from certain ergodic differential games,  more specifically, from the Mean Field Game (MFG) theory of
   J.M. Lasry and P.L. Lions \cite{LL-cs, LL-hf, LL-jj}. 
   These systems have been introduced to model differential games with a large number of players or agents with dynamics described by controlled diffusion processes, under simplifying features such as homogeneity of the agents and a coupling of Mean Field type. 
This allows to carry out a kind of  limit procedure as the number of agents tends to infinity which leads to simpler effective models.
Lasry and Lions have shown that for a large class of differential games (either deterministic or stochastic)   the limiting model reduces to   a  Hamilton-Jacobi-Bellman equation for the optimal value function of the typical agent coupled with a continuity (or Fokker-Planck) equation for the density of the typical optimal dynamic, the so-called Mean Field Game equations.  Solutions to these equations can be used to construct approximated Nash equilibria for games with a very large but still finite number of agents.
The rigorous proof of  the limit behaviour in this sense has been established by Lasry and Lions in~\cite{LL-cs,LL-jj} for ergodic differential games and extended by one of the authors
to several homogeneous populations of agents \cite{Fe13}. The time-dependent  case with nonlocal coupling has been addressed in a general context by \cite{clld, CarLoc}.
For a general overview on Mean Field Games, we refer the reader
to the lecture notes of 
Gu\'eant, Lasry, and Lions  \cite{GLL}, Cardaliaguet \cite{Card-notes}, the lecture videos
of P.-L. Lions at his webpage at  Coll\`ege de France, 
the first papers of Lasry and Lions~\cite{LL-cs, LL-hf, LL-jj} and of M. Huang, P.E. Caines, R.P. Malham\'e~\cite{HCM06}, \cite{HCM07},
the survey paper \cite{GS}, the  books by Gomes and collaborators~\cite{Gomesbook, gn}
and by Bensoussan, Frehse and Yan~\cite{BFY},
the two special issues \cite{BCC, BCC2}
 and the recent paper~\cite{clld} on the master equation and its application to the convergence of games with a large population to a MFG.
 For applications to economics see e.g. \cite{AGLL}, \cite{CarmonaBook}, \cite{gn}, \cite{GLL},
\cite{LL-sp}, \cite{LLLL}.
From the mathematical side, there are several important questions related to both the convergence and then the study of the limit MFG system itself, e.g. long time behaviour
\cite{Card-long-time,CLLP-long-time-average}, 
first order systems \cite{Card-fo-weak,Card-fo}, ergodic MFG systems \cite{Cir-stat-16,
BF16,gomi,gpv},
for time-dependent systems see also \cite{
gomes2015time},
 or homogenisation \cite{CDM}.  For further contributions see also  
\cite{achdou2017mean, ABC, Card-deg-diffusion, 
gomes2013continuous}. 
 The literature on Mean Field Games is very vast so the previous list is only partial and we refer to the references therein for a more extended bibliography. 
 	 
The novelty of this paper  consists in assuming that the dynamic of the average player is  a  diffusion of H\"ormander type and hence the  differential operators  arising in the system   are  degenerate: the second order operator  is  
not elliptic but 
only subelliptic.
Roughly speaking this means that the operators are elliptic 
only along  certain directions of  derivatives. Nevertheless the H\"ormander condition ensures that the Laplacian induced by these selected derivatives  is hypoelliptic. 
From the perspective of a single agent this means that the state cannot change in all directions, but the agent can move only along admissible directions: a subspace of the tangent space. This subspace depends on the state (position) of the agent.
Similarly the growth conditions on the Hamiltonian are restricted to some selected directions of derivatives.
 This extension  is not trivial and relies on recent deep achievements in the theory of H\"ormander operators and subellipitc quasilinear equations. When the known regularity  results will not be sufficient  to proceed, we  will use heat kernel estimates to overcome the problem. Moreover the techniques used here are  different from the standard elliptic case and can also be used in other contexts to gain a-posteriori regularity. \par
Hamilton-Jacobi equations  in the context of H\"ormander regularity have been extensively studied, see  e.g. \cite{{BCP}, {Capuzzo-Ishii},{Cutri},{DDMM},{DragoniHopf}}, in particular because of the intriguing connection between the PDE theory and the underlying geometry induced by the admissible directions.
This paper is to our knowledge the first one that connects these two recent and active areas.

We next state our main results:\\
\noindent 1 -  Under suitable assumptions (see Section 3)
and  assuming in particular that the Hamiltonian grows at most quadratically in the subgradient, we prove that there  exists a solution $(u, m) \in C_{\mX}^2(\T^d)  \times C(\T^d)$ of the system
$$
\left\{
\begin{aligned}
&\mL u + \rho u +  H(x, D_\Xc u) = V[m]   \\
&\mL^*m - \textrm{div}_{\Xc^* }( m g ( x, D_\Xc u ) )   = 0  \\
&\int_{\T^d} m \,dx  =1 ,
 \quad  m > 0,  
\end{aligned} \right. 
$$
where  $D_\Xc u$ is a subgradient associated to a family of H\"ormander vector fields (e.g. $D_\Xc u=\left(u_x-\frac{y}{2}u_z,u_y+\frac{x}{2}u_z\right)^T$ on $\R^3$ in the Heisenberg case) and $\mL $ is a hypoelliptic  operator, $\mL^* $ is the dual operator of  $\mL $ and  $\textrm{div}_{\mX^*} $ is the corresponding divergence operator.  Moreover by
$C_\Xc^2(\T^d)$ we indicate the sets of functions whose first and second derivatives in the selected directions exist and are continuous (see Section 2 for more formal definitions).\\
2 - Under suitable assumptions (see Section 4)
and assuming in particular that the Hamiltonian grows at most linearly in the subgradient, 
we prove that there exists a  solution $(\l, u, m) \in \R\times C_\Xc^2(\T^d)  \times C(\T^d)$ of 
the system 
$$
\left\{
\begin{aligned}
&\mL u + \l+  H(x, D_\Xc u) = V[m]   \\
&\mL^*m - \textrm{div}_{\mX^*} \big( m g ( x, D_\Xc u ) \big)   = 0  \\
& \int_{\T^d} u \, dx =0, \quad \int_{\T^d} m\, dx =1 ,
 \quad  m > 0 . 
\end{aligned} \right. 
$$
We  also show uniqueness for both the systems under standard monotonicity assumptions.\\
Those results are applied to the feedback synthesis of MFG solutions and of Nash equilibria of a
 large class of $N$-player differential games. 

The paper is organised as follows:
in Section 2 we introduce the H\"ormander condition and the corresponding first and second order operators 
and we state several regularity results and estimates which will be key in the proofs of our main results.
 In Section  3 we show existence for a stationary MFG system for at most quadratic Hamiltonians by a fixed-point argument in the presence of a regularisation. 
In Section  4 we remove this regularisation for Hamiltonians of at most linear growth and prove our main existence result.
In the Appendix  we show the convergence of Nash-equilibria as motivation for the MFG system studied. Since these results are very well-known in the non degenerate case and they do not lead to any substantial technical difference in the H\"ormander case, we will omit the proofs, only reporting briefly the results.

{\bf Acknowledgments:} 
The authors were supported by the EPSRC Grant ``Random Perturbations of ultra-parabolic PDEs under rescaling''.
The authors would like to thank  Nicolas Dirr and Pierre Cardaliaguet for the many interesting conversations and suggestions.

\section{Preliminaries and notations}
Let us consider $x\in \T^d$ the $d$-dimensional torus and $\mX=\{X_1, \ldots, X_m \}$ a family of smooth vector fields defined on $\T^d$ 
 satisfying the H\"ormander condition, 
i.e. 
\begin{equation}
\label{(I)}
Span\bigg(\mathcal{L}\big(X_1(x),\dots,X_m(x)\big)\bigg)=T_x \T^d \equiv \R^d,\quad \forall\, x\in \T^d,
\end{equation}
where $\mathcal{L}\big(X_1(x),\dots,X_m(x)\big)$ denotes the Lie algebra induced by the given vector fields and by $T_x\T^d$ we denote the tangent space at the point $x\in \T^d$. For more details on H\"ormander vector fields  we refer    to \cite{montgomery}.
Given a family of vector fields  $\mX=\{X_1, \ldots, X_m \}$ and  $u:\T^d\to \R$, we define:
\begin{align}
&\label{gradient} D_{\Xc} u=(X_1 u,\dots,X_m u)^T\in \R^m,\\
&
\label{sub_laplacian} \mL u = - \frac{1}{2} \sum_{j=1}^m X_j^2u\in \R.
\end{align}
For any vector-valued function $g:\T^d\to \R^m$, we will consider the divergence induced by the vector fields  $\mX=\{X_1, \ldots, X_m \}$, that is
\begin{equation}
\label{DIVERGENCE} 
  \textrm{div}_{\Xc}g=X_1\,{g_1}+\dots +X_m \,{g}_m,
\end{equation}
  where  $g_i$ indicates the $i$-component of $g$, for $i=1,\dots,m$.
  In particular,  later on, we will consider the divergence $
  \textrm{div}_{\Xc^*}g
$ induced by the dual vector fields $X^*_i=-X_i-\textrm{div}X_i$ where $\textrm{div}X_i$ indicate the standard  (Euclidean) divergence of the vector fields $X_i:\T^d\to \R^d$, for $i=1,\dots,m$.
Given the family of vector fields $\mX=\{X_1, \ldots, X_m \}$ we recall that any absolutely continuous curve $\gamma:[0,T]\to \T^d$ is called horizontal (or admissible) if there exists a measurable function $\alpha:[0,T]\to \R^m$ such that
\begin{equation}
\label{horizontalCurves}
\dot{\gamma}(t)=\sum_{i=1}^m\alpha_i(t)X_i(\gamma(t)),\quad \textrm{a.e.}\; t\in (0,T),
\end{equation}
where $\alpha_i(t)$ is the $i$-component of $\alpha(t)$ for $i=1,\dots,m$.\par
For all horizontal curves it is possible to define the length as:
$$
l(\gamma)=\int_0^T\sqrt{\sum_{i=1}^m\alpha_i^2(t)}\;dt.
$$
The {\em Carnot-Carath\'eodory distance} induced by the  family $\mX=\{X_1, \ldots, X_m \}$ is  denoted by $d_{CC}(\cdot,\cdot)$, and defined as
$$
d_{CC}(x,y)=\inf\left\{
l(\gamma)\,|\, \gamma \; \textrm{satisfying \eqref{horizontalCurves} with}\; \gamma(0)=x, \gamma(T)=y
\right\}.
$$
The H\"ormander condition implies that the distance $d_{CC}(x,y)$ is finite and continuous w.r.t. the original Euclidean topology induced on $\T^d$ (see e.g. \cite{montgomery}).
It is also known that  
 there exists $C> 0$ such that
\begin{equation}
\label{local_estimates_distances}
C^{-1}|x-y|\le d_{CC}(x,y) \le C |x-y|^{1/k}   
\end{equation}
 for all $x,y \in \T^d$, where $k\in \N$ is the  step, i.e. the maximum of the degrees of the iterated brackets occurring in  the 
fulfillment of the  H\"ormander condition, see \cite{NSW86}. 
It was proved in \cite[Lemma~5]{Sa-Ca84} and independently in \cite{NSW86} that there exists some $Q>0$, called the {\em homogenous dimension}, such that, for all $\d >0$ sufficiently small  
and for some $C>0$, 
\[
 C^{-1} \delta^Q \le  |B_{d_{CC}} (x, \d) | \le C \delta^Q,
\] 
for all $x\in \T^d$, where $B_{d_{CC}} (x, \d)$ is the ball of centre $x$ and radius $\delta$ w.r.t. the distance 
$d_{CC}$  and, for any $B\subset \T^d$, $|B|$ denotes the standard Lebesgue measure  
of $B$. 

\subsection{H\"older spaces and H\"older regularity  estimates}
Next we recall the definition of H\"older and Sobolev spaces associated to the  
 family of vector fields $\mX$ (we refer to \cite{Xu-sem-94} and \cite{XuZu97} for more details on these spaces).
For every multi-index $J=(j_1 , \ldots, j_m) \in \Z_+^m$ let 
$\mX^J = X_{j_1} \cdots X_{j_m}$. The {\em length} of a multi-index $J$ is $|J|= j_1+ \cdots +j_m$, thus 
$\mX^J$ is a linear differential operator of order $|J|$.  
For $r\in \N$  and $\alpha \in (0, 1)$ we define the function spaces 
\[
C_{\mX}^{0, \alpha}(\T^d) = \left\{  u \in L^{\infty }(\T^d) \; : \; \sup_{ \substack{x, y \in \T^d \\x\neq y} } \frac{|u(x) - u(y)|  }{d_{CC}(x,y)^\alpha } <\infty  \right\}, 
\]
\[
C_{\mX}^{r, \alpha}(\T^d) = \left\{  u \in L^{\infty }(\T^d) \; : \; \mX^J u \in C_{\mX}^{0, \alpha}(\T^d) \;\; \forall |J| \le r\right\}. 
\]
For any function $u\in C_{\mX}^{0, \alpha}(\T^d) $ one can define a seminorm as
$$
[ u ]_{C_{\mX}^{0, \alpha}(\T^d)}= \sup_{ \substack{x, y \in \T^d \\x\neq y} } \frac{|u(x) - u(y)|  }{d_{CC}(x,y)^\alpha },
$$
and, for every  $u\in C_{\mX}^{r, \alpha}(\T^d)$, the norm is defined as 
$$
\norma{u}_{C_{\mX}^{r, \alpha}(\T^d)}=  \norma{u}_{L^\infty(\T^d)}  +\sum_{1\leq |J|\leq r} [\mX^J u
]_{C_{\mX}^{0, \alpha}(\T^d)}.
$$
Endowed with the above norm, $C_{\mX}^{r, \alpha}(\T^d) $ are Banach spaces for any $r\in \N$ and $\alpha\in (0,1)$. \par
From estimates \eqref{local_estimates_distances}, it follows immediately  
\begin{equation}
\label{Holder-relations}
C^{-1}\norma{u}_{C^{0, \frac{\alpha}{k}}(\T^d)} \leq \norma{u}_{C_{\mX}^{0, \alpha}(\T^d)}\leq C\norma{u}_{C^{0, \alpha}(\T^d)}\quad
\implies
\quad
C_{}^{0, \alpha}(\T^d)\subset
C_{\mX}^{0, \alpha}(\T^d)
\subset
C_{}^{0, \frac{\alpha}{k}}(\T^d),
\end{equation}
where $\norma{u}_{C^{0, \alpha}(\T^d)}$ is the standard  H\"older norm, $k$ is the step in the H\"ormander condition and  $C>0$ is a global constant depending only on the dimension $d$ and the family of vector fields $\mX=\{X_1,\dots,X_m\}$.
More in general, for all $r\in \N$, 
$
C^{r, \alpha}(\T^d)\subset
C_{\mX}^{r, \alpha}(\T^d).
$\\
Let $r$ be a non-negative integer and $1\le p \le \infty $. We define the space 
\[
W_{\mX}^{r, p}(\T^d) = \left\{  u \in L^p(\T^d) \; : \; \mX^J u \in L^p(\T^d) , \; \forall J \in \Z_+^m, \; |J| \le r  \right\}\,.
\]
Endowed with the norm
$
\norma{u}_{ W_{\mX}^{r,p} (\T^d)     }
= \left( \sum_{|J|\le  r }  \int_{\T^d } 
|\mX^J u |^p \, dx  \right)^{1/p} ,   
$
$W_{\mX}^{r, p}(\T^d)$ is a Banach space.  For $p=2$ we write  $H_{\mX}^{r}(\T^d)$ instead of $W_{\mX}^{r, p}(\T^d) $ and in this case  
the space is Hilbert  when endowed with the corresponding inner product.  
Moreover, for any $1\le p< \infty$, the embeddings
\[
C_{\mX}^{kr, \alpha}(\T^d) \hookrightarrow  C^{r, \frac\alpha{k}}(\T^d) \,,
\]
\[
W_{\mX}^{r, p}(\T^d) \hookrightarrow W^{r/k, p}(\T^d) \,,
\] 
 hold true.  The first is proved in \cite{Xu92} and the second in \cite{Xu90-var}.\par 
In proving one of our main results we will also need the following compact embedding. 

\begin{lem}\label{comp-emb}
 $W_{\mX}^{1, p}(\T^d)$
is compactly embedded into $L^{p}(\T^d)$. 
\end{lem}
This follows from the previous embedding and the fact that the 
fractional Sobolev space  $W^{k/m, p}(\T^d)$ is compactly embedded into $L^p(\T^d)$ (see e.g. \cite{DPV12}).  \par
Next we want to  recall some H\"older  regularity results for linear and quasilinear subelleptic PDEs, key  for the later existence results. 
H\"older and Schauder estimates for subelliptic linear and quasilinear equations have been proved by Xu~\cite{Xu90-var,Xu-sem-94}, Xu-Zuily~\cite{XuZu97} and \cite{Lu96}; see also the references therein.
In particular we will consider the results  proved in~\cite{XuZu97}, but we will rewrite them in a stronger form, by combining them with some $L^p$-estimates proved by Sun-Liu-Li-Zheng~\cite{SLLZ16}.
The results in \cite{XuZu97} are proved for subelliptic systems but we will apply them
to the case of a single equation. 
We first consider linear equations of the form:
 \begin{equation}
 \label{linearPDE_Xu} 
 \textrm{div}_{\mX^*}\big(A(x)D_{\mX} u\big)+ g(x) \cdot D_{\mX} u+c(x)\,u =f(x)
 .
  \end{equation}
and assume that 
 \begin{equation}
 \label{Assumption on A}
A(x)\; \textrm{is a $m\times m$-uniformly elliptic matrix}.
  \end{equation}
  Note that in  the case of the sub-Laplacian the previous assumption is trivially satisfied since  $A(x)$ is equal to the identity $m\times m$-matrix.



\begin{thm}[$C^{2,\alpha}_{\mX}$-regularity for linear subelleptic PDEs,  \cite{XuZu97,SLLZ16}.]
\label{Holder_regularity_linear}
Assuming \eqref{Assumption on A} and that all coefficients of $A(x)$, $g(x)$, $c(x)$ and $f(x)$ are  H\"older continuous, then any  weak solution $u\in H_{\mX}^1(\T^d)$ of \eqref{linearPDE_Xu} belongs to $C_{\mX}^{2,\alpha}(\T^d)$ for some $\alpha\in (0,1)$.\par
Moreover  there exists a constant $C>0$ (depending only on the H\"older norms of the coefficients of the equation, on  $d$ and on the vector fields $\mX$) such that
$$
\norma{u}_{C_{\mX}^{2, \alpha}(\T^d)}\leq C.
$$
\end{thm}
\begin{proof}
First we recall that, if the coefficients are $C^{0,\alpha}$ then they are also $C^{0,\alpha}_{\mX}$ (see   \eqref{Holder-relations}).
Then Theorem 3.4 and Theorem 3.5 in   \cite{XuZu97} ensure that, given any $u$ weak $H_{\mX}^1$- solution,  $u$ belongs to $ C_{\mX}^{2,\alpha}(\T^d)$,  and the $C_{\mX}^{2,\alpha}$-H\"older norm  of $u$ is bounded by a constant depending on the H\"older norms of the coefficients, on the geometry of the problem (i.e. the step $r$, the dimension $d$ and the number of vector fields   $m$), but also on a constant $M$ such that
$
\norma{u}_{H_{\mX}^{1}(\T^d)}
\leq M.
$\\
We can now use the uniform $L^p$ estimates proved in  Theorem 1.4 in \cite{SLLZ16} to show that the constant $C$ is actually independent of $M$, i.e. independent of  the $H^{1}_{\mX}$-norm of $u$. Note that H\"older regularity on a compact domain implies all the necessary $L^p$-bounds to apply the result in \cite{SLLZ16}.
\end{proof}
 Let us now consider a subelliptic quasilinear equation of the form:
 \begin{equation}
 \label{semilinearPDE_Xu} 
 \textrm{div}_{\mX^*}\big(A(x)D_{\mX} u\big)=f(x,u,D_{\mX}u).
  \end{equation}
  and  assume that $f(x,z,q)$ is a  H\"older function with at most quadratic grow, i.e.
 \begin{equation}\label{quadratic-grow-f}
 |f(x,z,q)|\leq a|q|^2+b,
 \end{equation}
 for some non-negative constants $a$ and $b$.
\begin{thm}[$C^{1,\alpha}_{\mX}$-regularity for quasilinear subelleptic PDEs,  \cite{{XuZu97},{SLLZ16}}.]
\label{Holder_regularity_semilinear}
Assuming \eqref{Assumption on A}, \eqref{quadratic-grow-f} and that all the coefficients of the equation are  H\"older continuous,  then any  weak solution $u\in H_{\mX}^1(\T^d)\cap C(\T^d)$ belongs to $C_{\mX}^{1,\alpha}(\T^d)$ for some $\alpha\in (0,1)$ and there exists a constant $C>0$ (depending only on the H\"older norms of the coefficients of $A(x)$ and of $f$, on $a$ and $b$ in \eqref{quadratic-grow-f}, on the step $r$, on $d$ and $m$) such that
$$
\norma{u}_{C_{\mX}^{1, \alpha}(\T^d)}\leq C.
$$
\end{thm}
\begin{proof}
Combining once again the $L^p$-estimates in \cite{SLLZ16} 
with Theorem 4.1 in   \cite{XuZu97} one can immediately deduce the result.
\end{proof}
\begin{thm}[$C^{\infty}$-regularity, Theorem 4.2, \cite{XuZu97}]
\label{smooth_regularity_semilinear}
Under the assumptions of Theorem \ref{Holder_regularity_semilinear}, if in addition all coefficients in equation \eqref{semilinearPDE_Xu}  are $C^{\infty}(\T^d)$ then $u\in C^{\infty}(\T^d)$.
\end{thm}


 \section{Discounted systems with at most quadratic Hamiltonians}

In this section we consider a subelliptic MFG system with a first order nonlinear term that grows at most quadratic w.r.t. the horizontal gradient.
We assume:

\begin{itemize}


  
\item[{\bf (II-Q)}]  
For $q=\sigma(x)p\in \R^m$   there exists a constant $C \ge 0$ such that
\bel{H-qu-gr}
|H(x, q )| \le C( |q |^2 +  1   ) \quad \forall x\in\T^d\,,\; q \in \R^m .
\eeq
\end{itemize}

\begin{itemize}
\item[{\bf (III)}] The vector-valued function 
$
g\colon \T^d \times \R^m \to \R^m
$ is 
 H\"older-continuous. \par
(Note that since $\T^d$ is compact and we will later prove global bounds for $D_{\mX} u$,  the continuity of $g$ implies also that $g$ is globally bounded).

\item[{\bf (IV)}] Set $\Ac:=\left\{ m \in C(\T^d) \; : \; m>0 \,, \; \int_{\T^d} {m(x) \, dx =1 } \right\}  $, then the map 
$
V\colon \Ac \to L^\infty(\T^d)$  is assumed continuous and bounded.
Moreover, we assume that $V$ is {\em regularising}, that is, $V[m] \in C^\alpha_\mX(\T^d)$
for all $m \in \Ac$, and 
$
\sup_{m \in \Ac} \|V[m]\|_{C^\alpha_\Xc(\T^d)} < \infty. 
$
\end{itemize} 


\begin{thm}\label{dsqg}
Assume  \eqref{(I)},  {\bf (II-Q)}, {\bf (III)}, {\bf (IV)} and that $H(x,q)$ is locally H\"older,
 then given $\mL$  defined in \eqref{sub_laplacian}  with dual operator $\mL^*$ and 
 $\mathrm{div}_{\mX^*}$ defined as in \eqref{DIVERGENCE}   w.r.t. the dual vector fields $X^*_i=-X_i-\textrm{div}X_i$,  for every $\rho >0$ the system 
 \bel{sys-rho}
\left\{
\begin{aligned}
&\mL u + \rho u +  H(x, D_\Xc u) = V[m]   \\
&\mL^*m - \textrm{div}_{\Xc^* }( m g ( x, D_\Xc u ) )   = 0  \\
&\int_{\T^d} m \,dx  =1, 
 \quad  m > 0
\end{aligned} \right. 
\eeq 
has a solution $(u, m) \in C_{\mX}^2(\T^d)  \times C(\T^d)$. (Note that $u$ solves the system in the classical sense while $m$ is a weak solution in the distributional sense.)
\end{thm}
To prove the existence for the system  \eqref{sys-rho} we  need to look at both the equations involved, starting first from the associated linear PDE for $u$.
\begin{lem}\label{ll-dp}
 Assume \eqref{(I)}  and  that $\mL $ is the corresponding sub-Laplacian defined in \eqref{sub_laplacian}, 
 then for every   $\rho >0$ and  $f\in C_{}^{0,\alpha}(\T^d) $ 
 \begin{equation}
\label{ll-dp-eq} 
\mL u + \rho u = f   \text{ in }  \T^d 
\end{equation} 
has a unique solution $u \in C_\Xc^{2, \alpha } (\T^d) $. Moreover  $\exists\;C \ge 0$ (independent of $u$ and $f$) such that
\begin{equation}\label{stima}
\|u\|_{ C_\Xc^{2, \alpha } (\T^d) } \le C\, \|f\|_{C_\Xc^{0,\alpha } (\T^d)}.
\end{equation}
\end{lem}
\begin{proof}
The solution is unique by the strong maximum principle of Bony~\cite{Bo69} (see also 
Bardi and Da Lio~\cite{BL03}). 
We show the existence by vanishing viscosity methods, i.e. for all $\varepsilon>0$ we consider the operator 
$\mL_\e = - \e \Delta + \mL$, with $\e>0$ and the corresponding problem \eqref{ll-dp-eq}, replacing $\mL$ by $\mL_\e$.
Note that  $\mL_\e u+\rho u=f$ is a linear uniformly elliptic equation. 
It is well-known that such a problem has a unique classical solution $u_\e$, which is 
of class $C^{2,\alpha}$ since $f\in C^{0,\alpha}$  (see e.g. \cite[Lemma~2.7]{BF16}
and \cite{GT}).
Moreover
$
\|u_\e\|_\infty \le \frac1\rho \|f\|_\infty \,.
$
This implies that (up to a subsequence)  
$u_{\e} \to u$ in the weak$^*$-topology of $L^\infty(\T^d)$. 
Therefore $u$ is a distributional solution of $\mL u+\rho u =f$. 
Furthermore, if $f$ is smooth 
 then, by  H\"ormader's hypoellipticity Theorem~\cite{Ho67}, $u$ is smooth.
So let us assume  for the moment  that $f\in C^{\infty}(\T^d)$; then $u$ is in particular a classical solution satisfying the assumption of Zuily and Xu \cite{XuZu97}, thus Theorem \ref{Holder_regularity_linear}
 gives directly estimate 
\eqref{stima}. \\
If $f\in C_{\mathcal{X}}^{0,\alpha}(\T^d) $,  
one can bypass this obstacle by mollifications and  noticing that 
 estimate~\eqref{stima} is stable w.r.t. the mollification parameter.  
More precisely, when $f$ is not smooth but only H\"older, we introduce 
 $f_\zeta := f * \varphi_\zeta$, where $\varphi_\zeta(x) := \zeta^{-d} \varphi( x/\zeta)$
for $\zeta>0$ and $x\in \R^d$, and $\varphi$ is a mollification kernel, that is, 
 a nonnegative function of class $C^\infty$, with support in the unit ball of $\R^d$ and
$\int_{\R^d} \varphi(x)\,dx =1$. One can easily  check  that  
$f_\zeta \to f$ as $\zeta \to 0$
in $C_{\mathcal{X}}^{0,\alpha}(\T^d)$.   
Let $\{\zeta_n\}_{n\in \N}$ be a sequence of positive numbers converging to zero.  
For every $n\in \N$ there exists a unique solution
$u_n \in  C^{2,\alpha}(\T^d)$ to \eqref{ll-dp-eq} for $f= f_n\colonequals f_{\zeta_n}$, and by estimate~\eqref{stima},
we have $\|u_n- u_m \|_{ C_\Xc^{2, \alpha } (\T^d) } \le C\, \|f_n-f_m\|_{C_\Xc^{0,\alpha } (\T^d)}$  for some constant $C>0$ that does not depend on
$n, m \in \N$.  
Thus $\{u_n\}_{n\in \N}$ is Cauchy in $C_\Xc^{2, \alpha } (\T^d)$ (using  $f_n \to f$ in $C_{\mX}^\alpha(\T^d)$), hence it converges to some
$u$  in $C_\Xc^{2, \alpha } (\T^d)$. Passing to the limit as $n\to \infty$ in the equation $\mL u_{n} + \rho u_{n} =f_n$ 
and in the estimates $\| u_n \|_{C_{\mX}^{2, \alpha}(\T^d )  } \le C \|f_n\|_{C^\alpha_{\mX}(\T^d) }$
 we find that $u$ is a solution to~\eqref{ll-dp-eq}
and that estimate~\eqref{stima} is satisfied. 
\end{proof}
The existence and uniqueness for the subelliptic linear equation for $m$ is more technical. We first recall some heat kernel estimates and an ergodic result which will be key for the later results.
Consider the Cauchy problem 
\begin{equation}
\label{Solution-z}
\left\{ \begin{aligned}
&\frac{\partial z}{\partial t} - \mL z - g \cdot D_{\Xc} z = 0 \\
&z(0,x) = \phi(x)
\end{aligned}\right.
\end{equation}
 where $\phi$ is Borel and bounded and
 $g$ is H\"older-continuous. 
Then we have the following representation for the unique solution of \eqref{Solution-z}:
\[
z(t,x) = \int_{\T^d}K(t, x,y) \phi(y) \, dy\,,  
\] 
where the function $(t,x,y)\mapsto K(t,x,y)$, defined for $t>0$, $x,y\in\T^d$, $x\neq y$, 
is the heat kernel associated to the ultraparabolic operator $\partial_t - \mL - g \cdot D_{\Xc}$. 
We next recall some known Gaussian estimates satisfied by the heat kernel  $K(t,x,y)$: there exist 
constants $C=C(T)>0$ and $M>0$ (depending only on the H\"older norm of $g$)   such that 
\begin{equation}\label{heat-est}
\frac{C^{-1}}{|B_{d_{CC}}(x, t^{1/2}) |}e^{-M \,d_{CC}(x,y)^2/t } 
\le K(t,x,y) \le \frac{C}{|B_{d_{CC}}(x, t^{1/2}) |}e^{-M\, d_{CC}(x,y)^2/t } ,
\end{equation}
for all $T>t>0$ and $x\in \T^d$, where by $|B_{d_{CC}}(x, t^{1/2}) |$  we indicate the Lebesgue measure of the Carnot-Carath\'eodory ball centred at $x$ and of radius $R=t^{1/2}$.
This estimate has been firstly proved in the subelliptic case by \cite{JS86} for ``sums of squares'' operators on
compact manifolds and later generalised by many authors: 
in particular we refer to \cite{BBL10}.\par

  We now need to recall the following ergodic result.
	\begin{thm}[\cite{Be88}, Theorem II.4.1]\label{erg-thm}
Let $(S, \Sigma)$ be a compact metric space equipped with its Borel $\s$-algebra $\Sigma$. Let $P$ be a linear operator
defined on the Banach algebra of Borel bounded functions on $S$. 
We assume that 
$
\|P \| \le 1 $ and $ P(1) =1,
$	 and  there exists
 $\delta >0$ such that
\bel{cr-ass}
P\un_E(x) - P\un_E(y)  \le 1- \delta, \qquad \forall x,y \in S , \; E\in\Sigma \,,
\eeq
where by $\un_E(\cdot)$ we indicate the characteristic function of the Borel set $E$.\par 
Under these assumptions there exists a unique probability measure $\pi$ on $S$ such 
that 
\bel{erg-es}
\left| P^n \phi (x) - \int_S \phi \, d\pi \right| \le C e^{- k n } \|\phi \|_\infty \quad \forall x \in S\,,
\eeq
where $C= 2/(1-\delta)$, $k= - \ln{(1-\delta)}$. Then the measure $\pi$ is the unique {\em invariant measure} of the operator $P$, that is
the unique probability measure  satisfying
\[
\int_S P\phi \, d\pi = \int_S  \phi \, d\pi ,
	\]
	for every bounded Borel function $\phi$ on $S$. 
	\end{thm}
	The measure $\pi$ is called the {\em ergodic measure} of  the operator $P$ (for more details on ergodic measure see e.g.  \cite{DP06}).
	Property \eqref{erg-es}  is a ``strong'' ergodic property: it implies the convergence
		\[
	\lim_{n\to \infty}	P^n\phi  =   \int_S \phi \, d\pi \qquad \text{uniformly} 
		\]
		but also  provides  an exponential decay estimate on the convergence rate.

	\begin{rem}\label{DbC-rem}
	As noted also in \cite{Be88}, when applying the ergodic theorem 
		above usually one checks if the so-called Doeblin condition is satisfied. 
More precisely, we assume that $(S, \Sigma)$ is equipped with a probability measure $\mu$ and that $P$ has the form
		\[
		P\phi(x) = \int_S k(x,y) \phi(y) \, d\mu (y ) ,
		\]
		for some Borel and bounded kernel $k\colon S\times S \to \R$, 
		and that there exist a set $U$ with $\mu(U)>0$ and $\delta_0>0$
such that ({\em Doeblin condition}) 
\bel{DbC-est}
k(x,y) \ge \delta_0 >0 \qquad \forall x \in S, \; y \in U \,. 
\eeq
It is easy to check that \eqref{DbC-est} implies \eqref{cr-ass} with $\delta = \mu(U) \delta_0$.  In fact, using $S=\big(S\cap E\big) \cup \big(S\cap E^c\big) $:
$$
P\un_E(x) - P\un_E(y) 
=1-\int_S k(y,z)\un_{E}(z)\, dz-
\int_S k(x,z)\un_{E^c}(z)\, dz
\leq 1-\delta_0\bigg[\big|E^c\cap U\big|+\big| E\cap U\big|\bigg]\!\!\!=1\!\!-\delta_0|U|.
$$
	\end{rem}

\medskip 

Next we show existence and uniqueness for the weak solution of the subelliptic linear equation associated to  $m$.

\begin{lem}\label{adj-prob}
Assume \eqref{(I)} and that  $g\colon \T^d \to \R^m$ is H\"older continuous. 
Then the problem 
\bel{ad-pr-eq}
\left\{ \begin{aligned}
&\mL^* m - \textrm{div}_{\mX^*}(m g)  =0 \;\; \text{in } \T^d \,,\\
& \int_{\T^d} m \, dx =1\,,
\end{aligned} \right.
\eeq 
has a unique weak solution $m$ in $H_{\Xc}^1(\T^d)$. 
Moreover $0< \delta_0 \le m \le \delta_1$, for some $\delta_1, \delta_0$ depending only on the H\"older norm of $g$
and the coefficients of $\mL$ (i.e. the coefficient of the vector fields $X_1,\dots, X_m$). 
\end{lem}
  A solution $m$ of the PDE in \eqref{ad-pr-eq} is to be understood in the weak (or $H^1_\mX$) sense, i.e.  
we define the bilinear form 
\bel{bil-for}
\langle u, v  \rangle \colonequals \int_{\T^d }\left(-\frac12 \sum_{i=1}^mX_iu\cdot X^*_iu - (g \cdot D_{\mX} u)\, v   \right) d x
\eeq
and its dual $\langle u, v\rangle^*  \colonequals \langle v, u\rangle$ 
for all $u, v\in H_\mX^1(\T^d)$. Then $m$ is a {\em solution} of the PDE in \eqref{ad-pr-eq} if  
$\langle m, v\rangle^* =0$ for all $v\in H^1_{\mX}(\T^d)$.

\begin{proof}
The proof follows  the approach introduced in  \cite[Theorem 3.4]{BLP} 
for uniformly elliptic operators
and in \cite[Theorem~II.4.2
]{Be88}.
We want first to show that, for $\eta >0$ large enough and for every $\vf \in L^2(\T^d)$, the problem
\bel{aux-res}
\mL u - g \cdot D_{\mX} u  + \eta \,u = \vf  \,
\eeq 
is well-posed in $H_{\Xc}^1(\T^d)$ in the standard weak sense, that is
\[
\int_{\T^d}
  \left(-\frac12 \sum_{i=1}^mX_iu\cdot X^*_iu - (g \cdot D_{\mX} u)\, v  + \eta \,u\, v \right) d x
 =\int_{\T^d}\vf \, v\,dx, \quad \forall \; v\in C^{\infty}_0\big(\T^d\big).
\]
The previous well-posedness is proved by standard Hilbert space arguments.
In fact, on the space $H_{\Xc}^1(\T^d)$, we consider the bilinear form 
\[
\begin{gathered}
  \langle \cdot ,  \cdot \rangle_{\eta } \colon  H_{\Xc}^1(\T^d) \times H_{\Xc}^1(\T^d) \to   \R  \\
 (u, v) \mapsto  \langle u, v\rangle_{\eta} := 
\langle u, v\rangle + \int_{\T^d} \eta\, u\, v \,dx  
\end{gathered}
\] 
for all $u, v \in H^1_{\mX}(\T^d)$, where $\langle u, v \rangle$ is defined in~\eqref{bil-for}.
For $\eta>0$ large enough and for some $c_1 > 0$, $c_2 \ge 0$  we claim that for all $u,v \in H_{\Xc}^1(\T^d)$
\begin{align}
\langle u, u \rangle_{\eta} 
& \label{f-est}
\ge c_1 \| u\|^2_{ H_{\Xc}^1(\T^d) },\\
|\langle u, v \rangle_{\eta}| 
& \label{s-est} \le c_2 \| u\|_{ H_{\Xc}^1(\T^d) }  \| v\|_{ H_{\Xc}^1(\T^d) }.
\end{align}
We first check  estimate \eqref{f-est}. Since $g$ and $\textrm{div} X_i$ are by assumption continuous, hence 
bounded on $\T^d$, there exists $M\ge 0$ such that $\|g\|_\infty \le M$ and $\|\textrm{div} X_i\|_\infty \le M$. 
Moreover
\begin{align*}
\langle u, u \rangle_{\eta} & = \int_{\T^d}  \bigg( \frac12 \sum_{i=1}^m |X_i u|^2 + \frac12 
  \sum_{i=1}^m X_i u \,  (\textrm{div}X_i) u   
	-  \big(g \cdot D_{\mX} u\big)u  + \eta \,u^2 \bigg) \, dx\\  
&\ge \int_{\T^d}  \bigg( \frac12 \sum_{i=1}^m |X_i u|^2 - \frac12 
  \sum_{i=1}^m |X_i u|  |\textrm{div}X_i| |u|    
	-  |g|  | D_{\mX} u |  |u|  + \eta \,u^2 \bigg) \, dx\,.
\end{align*}
Using the inequality $ab \le (1/4)a^2 +  b^2$ and recalling $  | D_{\mX} u |^2=\sum_{i=1}^m |X_i u|^2$,
we find 
\begin{align*}
\langle u, u \rangle_{\eta} &\ge   \int_{\T^d}  \bigg( \frac12| D_{\mX} u |^2 -  
  \frac18 | D_{\mX} u |^2   -  |\textrm{div}X_i|^2 |u|^2    
	-   \frac14   | D_{\mX} u |^2 -  |g|^2 |u|^2  + \eta \, u^2 \bigg) \, dx  \\
	&\ge  \frac18 \int_{\T^d}  | D_{\mX} u |^2 \, dx + 
	(\eta -  2M^2)  \int_{\T^d} u^2 \,dx   		\,,
\end{align*}
from which, taking $\eta > 2M^2$, we obtain the first estimate~\eqref{f-est} 
for a suitable $c_1 >0$ (in particular $c_1=\min\big\{1/8,\eta-2M^2\big\}>0$). 
For   estimate \eqref{s-est}
similarly
\[
|\langle u, v \rangle_{\eta}| \le \frac12 \int_{\T^d} |D_{\mX} u| |D_{\mX}v| \, dx 
+M \int_{\T^d} |D_\mX u| |v|\, dx + \eta \int_{\T^d} |uv| \, dx  \,,
\] 
and by the Cauchy-Schwarz inequality for integrals
\begin{align*}
|\langle u, v \rangle_{\eta}| \le
& \frac12 \| D_\mX u \|_{L^2(\T^d) } \| D_\mX v \|_{L^2(\T^d) }
+(M + \eta) \|u \|_{L^2(\T^d) } \|v \|_{L^2(\T^d) }
\\
\le &\left(M+\eta+\frac{1}{2}\right)\norma{u}_{H^1_{\mX}(\T^d)}\, \norma{v}_{H^1_{\mX}(\T^d)},
\end{align*}
where we have used simply $a\,b+c\,d\leq (a+c)\,(b+d)$ for every non-negative scalars $a,b,c$ and $d$. This 
gives \eqref{s-est} with $c_2 =\left(M+\eta+\frac{1}{2}\right) > 0$. 
 Then the claim is proved.\par
Thus the bilinear form $\langle\cdot,\cdot \rangle_{\eta}$
is coercive and continuous.
Clearly, 
\[
  H_{\Xc}^1(\T^d)   \ni u \mapsto    \int_{\T^d}  u \vf \, dx\, \in \R   
\]
is a continuous linear functional on $H_{\Xc}^1(\T^d)$. 
Therefore by the Lax-Milgram Theorem there exists a unique 
$u\in H_{\Xc}^1(\T^d)$ such that, for all $v \in H_{\Xc}^1(\T^d)$,
\[
\langle u, v \rangle_{\eta} = \int_{\T^d}\vf v \, dx.
\]
For every $\eta>0$ large enough (i.e. $\eta>2M^2$), we define the following  linear operator
$
T_\eta \colon L^2(\T^d) \to L^2(\T^d)\, 
$
by 
$T_\eta \vf:=u,$
 where $u$ is the unique solution to \eqref{aux-res}.\par
  Note that  $T_\eta \vf=u\in  H_\Xc^1(\T^d) \subset L^2(\T^d)$.
Since the embedding of $H_{\Xc}^1(\T^d)$ into $L^2(\T^d)$ is compact (see Lemma \ref{comp-emb}), $T_\eta$ is a  linear compact operator.
Thus the equation 
\[
\mL^* m - \textrm{div}_{\mX^*}(m g)  =0 \quad \text{in } \T^d 
\]
is equivalent to 
\bel{aux-res-b}
(I - \eta T_\eta)^* m  = 0\,,  
\eeq
where $I$ is the identity operator of $L^2(\T^d)$. \\
Since $T_\eta$ is compact,  the Fredholm alternative applies. 
Indeed 
$I - \eta T_\eta$ is a Fredholm operator of index zero (see e.g. \cite[Lemma~4.45]{AA02}). This means that the kernels of $I - \eta T_\eta$ and $(I - \eta T_\eta)^*$ have the same dimension, in other words, the number of
linearly independent solutions of the equation $(I - \eta T_\eta)^* =0 $ is equal to the number of linearly independent solutions of the
equation $I - \eta T_\eta =0$. 
Then we must find the number of linearly independent solutions of 
$
(I - \eta T_\eta) u= 0\,,
$
that is
\[
\mL u + g (x) \cdot D_\Xc u =0\,. 
\]
By \cite[Theorem~3.3]{XuZu97} the solution $u$ belongs to   $C^{2, \alpha}_\Xc(\T^d)$
for some $\alpha \in (0,1)$. 
Moreover the operator $\mL + g \cdot D_{\Xc}$  satisfies the strong maximum principle, see \cite{Bo69} and \cite{BL03}. 
Thus by the considerations above (Fredholm alternative) the equation \eqref{aux-res-b}, and hence \eqref{aux-res},
admits a unique solution $m\in H^{1}_\Xc(\T^d)$ up to a multiplicative constant.\par
The upper and lower bounds  for $m$ (that imply in particular the positivity of $m$) are shown by its interpretation as the ergodic measure 
of the diffusion having generator $\mL + g \cdot D_\Xc$. They rely on an ergodic theorem and on
the Gaussian estimates \eqref{heat-est}. 
In fact, using that $d_{CC}(x,y)$ is continuous on $\T^d$ (compact), we can easily see from~\eqref{heat-est} (by simply taking the maximum and the minimum of $d^2_{CC}(x,y)$ on $\T^d$) that there exits $\d_0, \d_1>0$ such that 
\begin{equation}
\label{boundsForM}
\d_0\leq K(1,x, y) \leq \d_1\qquad  \forall x,y \in  \T^d.
\end{equation}
Therefore we can apply Theorem~\ref{erg-thm} and Remark~\ref{DbC-rem} with $S=\T^d$, $\Sigma$ the Borel $\sigma$-algebra on $\T^d$, $\mu$
the Lebesgue measure on $\T^d$ and operator $P$ defined by
\[
P\phi(x) = z(1, x) = \int_{\T^d} K(1,x,y) \phi(y) \, dy \,. 
\]		
Note that $P^n\phi(x) = z(n, x)$. Then Theorem~\ref{erg-thm} implies the existence of a unique invariant  probability measure $\pi$
such that 
\bel{conv-z-pi}
\left| z(n, x) - \int_{\T^d} \phi(y) \, d \pi(y)  \right| \le C  e^{-kn } \|\phi\| \,.
\eeq	
	Using  $m$ defined as the unique solution of \eqref{ad-pr-eq} and $z(t,x)$ 
	defined as unique solution of \eqref{Solution-z}, we want to show the following claim:
\bel{id-z-m} 
\int_{\T^d} z(t, x) m (x) \, dx  =  \int_{\T^d} \phi(x) m(x) \, dx,  \; \forall \, t\geq 0.
\eeq 
To prove the previous claim, first note that for $t=0$ \eqref{id-z-m} is trivially satisfied by the initial condition.
We want to show that the right-hand side in \eqref{id-z-m} is constant in time, so we look at
$$
\frac{d}{dt} \int_{\T^d} z(t, \cdot) m  \, dx  =\int_{\T^d} \partial_t z m  \, dx  
=
\int_{\T^d} \big(\mL z +g\cdot D_{\mX} z\big)
\, m  \, dx=
\int_{\T^d} \big(\mL^*m- \textrm{div}_{\mX^*} (m\,g)\big) z\,dx
=0.
$$
Then $\int_{\T^d} z(t, x) m (x) \, dx =\int_{\T^d} z(0, x) m (x) \, dx$, for all $t\geq 0$, that proves  the claim  \eqref{id-z-m}.\par
By using \eqref{conv-z-pi} and taking $t=n$ in \eqref{id-z-m} and passing to the limit as $n\to +\infty$, we can deduce:
\[
 \int_{\T^d}  \phi(x) \,m (x)\, dx 
=\int_{\T^d} \bigg(\int_{\T^d} 
 \phi (x)\, d\pi (x)  \bigg) m(x)d x=\int_{\T^d} 
 \phi (x)\, d\pi (x) ,
\] 
for any Borel bounded function $\phi$ on $\T^d$ (where we have used $
\int_{\T^d} 
m =1
 $). Thus $m$ is the density measure 
of the probability measure $\pi$ and therefore $m \ge 0$  a.e. on $\T^d$. \par
Using \eqref{boundsForM} together with  \eqref{id-z-m} for $t=1$, it follows that 
\[
\delta_1 \int_{\T^d } \phi (y) \, dy \ge \int_{\T^d}  \phi(y) m(y) \, dy 
\ge \delta_0 \int_{\T^d } \phi (y) \, dy,
\]
for any bounded and Borel function $\phi \ge 0$ on $\T^d$. 
Since $\phi\ge 0$ is arbitrary, one can deduce 
$
\delta_0 \le m \le \delta_1 ,
$
thus Lemma~\ref{adj-prob} is proved. 
\end{proof}

We can now prove our first existence result for a subelliptic MFG system.

\begin{proof}[Proof of Theorem~\ref{dsqg}]
The proof is based on a corollary of  Schauder's fixed point theorem. 
 More precisely, we apply \cite[Theorem~11.3]{GT} 
which states that, if $T\colon \mB \to \mB$ is a continuous and compact operator in the Banach space 
$\mB$ such that the set 
$\{ u \in \mB \; : \; s T u = u ,\;  0\le s \le 1 \}$ is bounded, then 
$T$ has a {\em fixed point}, that is, there exists $u \in \mB$ such that $Tu =u$.
We define 
the Banach space 
$
\mathcal{B} = C_\Xc^{1, \alpha }(\T^d)\,, 
$
where $0<\alpha < 1$ is to be fixed later, and the
operator  
$
T \colon \mathcal{B} \to \mathcal{B} \,, 
$
 according to the scheme
$
v \mapsto m \mapsto  u \,.  
$ 
This means that, given $v \in \mathcal{B}$, we solve the second equation 
together with the corresponding conditions 
\[\left\{
\begin{aligned}
& \mL^*m - \textrm{div} ( m g ( x, D_\Xc v ) )   = 0 \,  \text{ in }\T^d\,,  \\
& \int_{\T^d} m \, dx =1 \,, \quad m>0 \, \text{ in } \T^d
\end{aligned}\right.
\]
and by Lemma~\ref{adj-prob} we find a unique solution $m \in H^1_\mX(\T^d)\cap L^\infty(\T^d)$. Moreover $m$ is bounded.
By assumption ${\bf (IV)}$, $V[\cdot]$ is regularizing, hence the function $f(x) = V[m](x) - H(x, D_\Xc v (x) )$ belongs to 
$C^\alpha_\Xc (\T^d)$.  Thus we apply Lemma~\ref{ll-dp} and deduce that
\begin{equation}\label{finally}
\mL u + \rho u +  H(x, D_\Xc v) =  V[m] 
\end{equation}
admits a unique solution $u \in C^{2}_\Xc(\T^d)$.  Set 
$T v = u,$
where $u$ is the unique solution of \eqref{finally},
 it is easy to check that $T$ is continuous and compact, using  that
$C_\Xc^2(\T^d)$ is compactly embedded into $C^{1,\alpha }_\Xc(\T^d)$ for all $\alpha \in (0,1)$.
Therefore,  in order to apply \cite[Theorem~11.3]{GT} 
we need to show
that \[
\mathcal{A}=\{ u\in \mathcal{B} \;: \;  \exists \,0\le s\le 1 \text{ such that } u = s T u\}
\]
is bounded in $C_\Xc^{1, \alpha }(\T^d)$.  So note
that:
 if $u$ is a fixed point of $sT$ (i.e. $sTu=u$), then it is also a solution of 
\begin{equation}
\label{fare}
\mL u + \rho u + s H(x, D_\Xc u) = s V[m] .
\end{equation}
Then looking at the minimum and maximum of $u$, we find
\[
\|u\|_{\infty, \T^d} \le \frac{s}\rho  \sup_{m\in H^1_{\Xc}(\T^d)} \|V[m] -H (\cdot, 0 )\|_{\infty, \T^d},
\]
which is finite since $V[\cdot]$ is by assumption bounded. \par
The key step is now to apply  $C_\Xc^{1, \alpha }$-regularity  for semilinear equation (Theorem \ref{Holder_regularity_semilinear}) that gives
\begin{equation}
\label{prelievo}
\| u \|_{C_\Xc^{1, \alpha } (\T^d) } < C 
\end{equation}
 for some constant $C>0$ and $\alpha \in (0,1)$ independ of $u$
and $s\in [0,1]$. \par
Note that, 
in order to apply the given theorem, we should write our equation in divergence form, which we can easily do by
using the relation $X_i^* = -X_i - \textrm{div} X_i$ 
(by adding the term  $-\sum_{j=1}^n (\textrm{div}X_j) X_j$ 
to the Hamiltonian). Observe that the new Hamiltonian has the same properties of the original Hamiltonian; in particular, it
grows at most quadratically in $D_\Xc u$ (in fact the functions $\textrm{div}  X_j$ are bounded due to the $C^{\infty}$-regularity of the vector fields 
$X_j$). 
Using estimate \eqref{prelievo} we can look at the semilinear PDE \eqref{fare} as a linear PDE with an H\"older right-hand side  $ f(x)=  s V[m] -s H(x, D_\Xc u)-\sum_{j=1}^m (\textrm{div}  X_j) X_ju 
$; hence we can 
apply the
 Schauder type result  for linear equations proved in
 \cite{XuZu97} (see Theorem \ref{Holder_regularity_linear}), that  implies 
$u \in C_\Xc^{2, \alpha } (\T^d) $ and
$
\| u \|_{C_\Xc^{2, \alpha } (\T^d) } < C.
$
To conclude we need only to remark that $\textrm{div}  X_j \in C^{0, \alpha}_\Xc(\T^d)$ (since the vector fields are smooth on a compact domain)  in order to apply the previous   $C_\Xc^{2, \alpha }$-estimates.  
\end{proof}

\section{Ergodic system with linear growth}

We now want to study the ergodic problem that can be obtained by letting $\rho \to 0^+$ in \eqref{sys-rho}.
However, to study this,  we need  a more restrictive assumption on the Hamiltonian, i.e. we assume that  $H$ grows at most linearly in $|D_\Xc u|$. More precisely:

\begin{enumerate}
\item[{\bf (II-L)}]  $H(x,q)=H(x,\sigma(x) p)$ grows at most linearly w.r.t. $q$,
i.e. $\exists\, C \ge 0$ such that 
\bel{H-li-gr}|H(x, q )| \le C( |q | +  1   ) \quad \forall x\in\T^d\,,\; q \in \R^m .\eeq
\end{enumerate}
We prove existence of solutions for the system of ergodic PDEs under condition 
{\bf (II-L)}.

\begin{thm}[Existence]\label{erg-thm-lin} 
Assume  \eqref{(I)},  {\bf (II-L)},  {\bf (III)},  {\bf (IV)} and that  $H(x,q)$ 
is locally H\"older, then the system 
 \bel{sys-erg}
\left\{
\begin{aligned}
&\mL u + \l+  H(x, D_\Xc u) = V[m]   \\
&\mL^*m - \textrm{div}_{\mX^*} \big( m g ( x, D_\Xc u ) \big)   = 0  \\
& \int_{\T^d} u \, dx =0, \quad \int_{\T^d} m\, dx =1 ,
 \quad  m > 0  
\end{aligned} \right. .
\eeq 
has a solution $(\l, u, m) \in \R\times C_\Xc^2(\T^d)  \times C(\T^d)$.  
\end{thm}
\begin{proof}
For $\rho >0$ let $(u_\rho, m_\rho ) \in C^2_\mX(\T^d)\times \big( H_\mX^1(\T^d)\cap L^\infty(\T^d) \big)$ be a solution of \eqref{sys-rho} by the existence result given in Theorem~\ref{dsqg}. 
Looking at the minima and maxima of $u_\r$, we have
\begin{equation}\label{bound}
 \|\r u_{\r} \|_{\infty} \le \sup_{m\in  H_\mX^1(\T^d)}\| H(\cdot, 0)- V[m]\|_{\infty} \,.
 \end{equation}
Let $<u_\r>:= \int_{\T^d} u_\r \,dx $ be the average of $u_\r$,
the key estimate is given in the following claim:
there exist $\rho_0>0$ and $C>0$ (independent of $\rho $) such that
\begin{equation}\label{stima-r}
\|u_{\r} - <u_{\r}>\|_{\infty} \le C, \quad \forall\; 0<\rho<\rho_0.
\end{equation}
To prove \eqref{stima-r} we adapt some ideas from \cite{AL98}. 
Assume by contradiction that there is a sequence $\r_n \to 0$ such that $\|u_{\r_n} - <u_{\r_n}>\|_{\infty}\to +\infty$ 
or equivalently, such that the sequence $$\e_n:=\|u_{\r_n} - <u_{\r_n}>\|_{\infty}^{-1}\to 0.$$ 
 Then the renormalised functions $\psi_n:=\e_n(u_{\r_n} - <u_{\r_n}>)$ satisfy
\begin{equation}
\label{sol_Pn}
\mL {\psi_n}+\e_n H\left(x, \frac{D_\Xc\psi_n}{\e_n}\right)+\r_n\psi_n=  \e_n(V[m_{\r_n}]  - \r_n<u_{\r_n}>) .
\end{equation}
We now apply \cite[Theorem~17]{Xu90-var} 
to deduce that the sequence $\{\psi_n\}$ is equi-H\"older continuous. In fact $\psi_n$ solve quasilinear equations of the same form as in \cite{Xu90-var} with
$A_i(x,u,\xi)=\xi_i$ and
$ B(x,u,\xi)=\varepsilon_nH\left(x,\frac{\xi}{\varepsilon_n}\right)
-\rho_n u-\varepsilon_n\big(V[m_{\rho_n}]\big)-\rho_n  <u>;
$
then it is easy to check that all conditions on the equation are satisfied just taking $g=0$, $f=1$ and a  $\Lambda$ depending only on the bound for $V[\cdot]$, the constant in   {\bf (II-L)} and the Lebesgue measure of $\T^d$.
Thus  \cite[Theorem~17]{Xu90-var} tells us that, taking $\r_n\leq 1$ and $\varepsilon_n\leq 1$, the H\"older norms of the solutions $\psi_{\r_n}$ are equi-bounded independently on $n$, which implies that $\psi_{\r_n}$ are equi-H\"older.
Therefore (up to a subsequence) we get that $\psi_n$ converges uniformly to a function $\psi$. 
Note that the functions $\psi_n$ are all renormalised, then $\|\psi\|_{\infty}=1$.
Moreover, since $\int_{\T^d} \psi_n \,dx=0$ by definition, then there exists a point $x_n\in \T^d$ such that   $\psi_n(x_n)=0$. Thus (up to a  further subsequence) we get $\psi(\ol x)=0$ for some  $\ol x\in \T^d$.
By using assumption  {\bf (II-L)} into equation
\eqref{sol_Pn}, one finds out that $\psi_{\r_n}$ are  classical (and hence viscosity) subsolutions of
\begin{equation}
\label{subsol_Pn}
\mL {\psi_n}-C \,|D_\Xc\psi_n|+\r_n\psi_n-  \e_n(V[m_{\rho_n}] - \r_n<u_{\r_n}>+C) =0.
\end{equation}
and  classical (and hence viscosity)  supersolutions of 
\begin{equation}
\label{supersol_Pn}
\mL {\psi_n}+C \,|D_\Xc\psi_n|+\r_n\psi_n-  \e_n(V[m_{\rho_n}] - \r_n<u_{\r_n}>-C) =0.
\end{equation}
 Finally by taking $n\to \infty$ in  \eqref{subsol_Pn}  and \eqref{supersol_Pn} and  by using the stability for viscosity subsolutions and viscosity supersolutions under uniform convergence (see e.g. \cite{BCD08}), 
$\psi$ is a viscosity subsolution of 
$
\mL {\psi} -C|{D_{\Xc}\psi}| = 0
$
and a viscosity supersolution of 
$
\mL {\psi} +C|{D_{\Xc}\psi}| = 0 
$.
Since $\mL$ is the subelliptic Laplacian associated to smooth H\"ormander vector fields and $\psi$ is periodic, we deduce from the strong maximum principle (see \cite{Bo69}, \cite{BL03}) that $\psi$ must be a constant, which  contradicts  $\|\psi\|_{\infty}=1$ and $\psi(\ol x)=0$,
proving thus \eqref{stima-r}.\par
We complete the proof of the theorem by showing that there exists a sequence $\r_n \to 0$ such that, for $w_\r:=u_{\r
} - <u_{\r}>$,
\begin{equation}
\label{conv-r}
\left(\r_n <u_{\r_n}>, \,w_{\r_n}, \,  m_{\r_n} \right) \to (\l, \, u, \, m) \qquad \mbox{in } 
 \R \times C_{\mX}^2(\T^d) \times  H_{\mX}^1(\T^d),    
\end{equation}
where $(\l, \, u, \, m)$ is a solution of~\eqref{sys-erg};
the convergence $m_{\rho_n} \to m$ is in the weak topology of $H_{\mX}^1(\T^d)$.
Indeed, we note that $(w_{\r}, m_{\r})$ solves 
\begin{equation}
\label{LLsys:r}
\left\{\begin{aligned}
&\mL {w_{\r}}+\r w_\r+H(x, D_\Xc {w_{\r}})=  V[m_{\r} ]  - \r <u_{\r}>  
\quad\mbox{in}\; \T^d,
\\ 
& {\mL}^*{m_{\r}}-\textrm{div}_{\Xc^*}\left(g(x, D_\Xc{w_{\r}} ){m_{\r}}\right)=0,
\\ 
&\int_{\T^d} m_\r(x) dx = 1,\quad m_\r>0 \,.
\end{aligned} \right. 
\end{equation}
By the a-priori  H\"older estimates for quasilinear subelliptic equations 
recalled in Theorem \ref{Holder_regularity_semilinear}  we know that
\[
\| {w_{\r}}\|_{C^{1, \alpha}_\Xc(\T^d)} \le C,
\]
for some $\alpha \in (0,1)$ and $C>0$  depending only on an upper bound of
 $\| {w_{\r}}\|_{\infty}$  and on the data of the problem, in particular on
the supremum norm of 
$V[m_{\r}]  - \r<u_{\r}>\,,$  
 which is bounded uniformly in $\r$ by  {\bf (IV)} 
and \eqref{bound}. In other words, $\alpha$ and $C$ can be   chosen independent of $\r$. 
Next by Schauder local estimates for subelliptic linear equations \cite[Theorem 3.5]{XuZu97}, we have
\begin{equation}
\label{est-w}
\|{w_{\rho}}\|_{C_\Xc^{2,\alpha }(\T^d
)} \le C,
\end{equation}
for some $C$ and $\a$ independent of 
  $\rho$.
On the other hand, by Lemma~\ref{adj-prob} 
and assumption~{\bf (III)} 
\begin{equation}
\label{est-m}
\|m_{\r }\|_{H_\Xc^{1}(\T^d)} \le C,
\end{equation}
for $C\ge 0$ independent of small enough $\r$.
Since $C_\Xc^{2, \alpha}(\T^d)$ 
is compactly embedded into $C_\Xc^2(\T^d)$,   the previous estimates  \eqref{est-w}, \eqref{est-m}
and the fact that the set $\{\r <u_\r> \; : \; \r  >0 \} $ is bounded (in $\R$) by \eqref{bound}, 
we can extract a sequence $\r_n \to 0$ 
such that  ~\eqref{conv-r}  holds.
Furthermore, since $g$ is locally H\"older by assumption {\bf (III)}
	and $D_\Xc w_{\rho_n } \to D_\Xc u$ in  $C_{\mX}^1(\T^d)$, then 
		$g_{n} \colonequals g(\cdot, D_\Xc w_{\rho_n }(\cdot ) ) \to 
		g(\cdot, D_\mX u(\cdot ))$. 
		Let $\langle \cdot, \cdot \rangle_n$
	denote the bilinear form associated with $g_n$ in the same fashion as 
	$\langle\cdot, \cdot\rangle$ denotes the bilinear form associated with $g$ after the  statement of Lemma~\ref{adj-prob}.
	Since  $m_{\rho_n}$ is the solution of the second equation in \eqref{LLsys:r},
	$\langle m_{\rho_n}, \vf\rangle_n^*=0$ for all $\vf \in H_\mX^1(\T^d)$. 
	From this and the fact that $g_n \to \bar g(\cdot,  D_\mX u (\cdot ))$ in $L^2(\T^d)$, it is fairly easy to deduce that
	$\langle m , \vf \rangle^* =0$ for all $\vf \in H^1_\mX(\T^d)$. Thus $m$
	is a solution of the  second equation in~\eqref{sys-erg}. The normalising conditions in the third row of 
	\eqref{sys-erg} are clearly preserved in the  limit. 
Thus the triplet $(\lambda, u, m)$ is indeed a solution of \eqref{sys-erg}.
	\end{proof}




Exactly as in the elliptic case, both the previous MFG systems have  unique solutions under suitable monotonicity assumptions.\par
 Recall that an 
 operator $V$, defined on some subset of $L^2(\T^d)$ with values in 
$L^2(\T^d)$, 
 is  {\em monotone}
if 
$
 \int_{\T^d} \left(V[m_1]-V[m_2] \right) (m_1-m_2)\, dx \ge 0$, $ \forall m_1, m_2,
$
and it is {\em strictly monotone} if the   inequality  is strict 
for all $m_1 \neq m_2$. 
Given a function $H \colon \T^d \to \R$ and a vector-valued map $g\colon  \T^d \to \R^m$, we say
that $H$ is {\em $g$-convex} if 
$
H( q_2) - H(q_1) - g(q_1) \cdot (q_2-q_1) \le 0, 
$
for all $q_1, q_2 \in \T^d$. 
If the   inequality  is strict for $q_1\neq q_2$, $H$ is {\em strictly $g$-convex}.  
\begin{thm}[Uniqueness]\label{uniq-thm}        
Assume that  one of the two following assumptions holds:
\begin{description}
\item[(i)]
$V$ is monotone in $L^2$ and $ H$ is  strictly $(-g)$-convex, or
\item[(ii)]
 $V$ is   strictly monotone in $L^2$
 and  $H$ is  $(-g)$-convex.
\end{description}
Then  the system  \eqref{sys-erg} has a unique weak solution.
\end{thm}
The proof is standard so we omit it.
\begin{rem}$\,$
\begin{enumerate}
\item Hamiltonians $H$ coming from optimal control are ``$(-g)$-convex'' and, under suitable assumptions, strictly $(-g)$-convex. 
\item The strict $(-g)$-convexity can be relaxed  requiring that
$
H( q_2) - H(q_1) - g(q_1) \cdot (q_2-q_1) \le 0, 
$
 implies $g(x, q_1)= g(x, q_2)$, instead of $q_1=q_2$. In this way one can cover also the case
$H(x, q)= |q|$ and $g(x, q) = -q/|q|$ for $q\neq0$, $g(x, 0)=0$. 
\item Similarly one can state  the uniqueness for the ``discounted'' system~\eqref{sys-rho}. 
\end{enumerate}

\end{rem}

	


\section{Appendix}
Here we want to show  briefly  some applications 
 to stochastic differential games, which motivate the study of our MFG system.
 These applications are  standard and the H\"ormander  degenerate case is   similar  to the known uniformly elliptic case. We include them for completeness but  omitting all details and proofs.

\subsection{The optimal-control-fixed-point problem of MFG theory}

The heuristics of Mean Field Games leads to a mathematical problem that 
consists of an optimal control problem followed by a fixed point problem. 
This heuristics  is well explained in the literature, for example, in~\cite{LL-cs}, \cite{LL-jj}, 
\cite{Card-notes} and~\cite{ABC}. 
 Moreover, the relation between $N$-player games and Mean Field Games has been considered in the literature since the very beginnings, \cite{LL-cs,LL-hf,LL-jj}, \cite{HCM06, HCM07}; see also 
\cite{Fe13} where results have been extended to several homogeneous populations of agents
for ergodic problems.
Recently the  time-dependent case has been addressed by Cardaliaguet, Delarue, Lasry and Lions in  \cite{clld}. 
For the first mathematical problem hinted above,
we give a short self-contained description based mainly on 
 the notes of Cardaliaguet~\cite{Card-notes}
 and some comments in the introduction of Araposthatis et al.~\cite{ABC}.\par
As in the previous section, we consider a family of $m$ smooth vector fields $\Xc= \left\{ X_1, \ldots, X_m \right\}$, $m\in\N$, satisfying the H\"ormander condition \eqref{(I)} on the
$d$-dimensional  torus $\T^d$,  
and a map $V\colon P(\T^d) \to P(\T^d)$ that satisfies condition {\bf (IV)}.
Let $W_t=(W_t^1, \ldots, W_t^m)$ be an $\R^m$-valued Brownian motion in a complete filtered probability space
$
\big(\Omega, \mF, (\mF)_{t\in \R_+}, \mathbb{P}  \big).
$
The Brownian motion is assumed adapted with respect to the filter $(\mF)_{t\in \R_+}$
and the filter is required to satisfy 
the so-called {\em standard assumptions}, see e.g.~\cite{Be88}, \cite{Pr05}.  \par
We consider the stochastic differential equation 
 \bel{dyn}
\left\{\begin{aligned}
d\xi_t &= \sum_{k=1}^m b^k(\xi_t, \alpha_t  )  X_k(\xi_t)   dt +   \sum_{k=1}^m  X_k(\xi_t) \circ  dW^k_t\,, \\
\xi_0 &= x_0\,,
\end{aligned}\right.
\eeq
where the notation ``$\circ$'' denotes Stratonovich integration and $x_0\in\T^d$ is some fixed {\em initial condition}. 
The {\em  controls}  $\a= (\a^1, \ldots, \a^m)\colon [0,\infty) \times \Omega \to A$ 
are measurable $(F_t)_{t \ge 0 }$-adapted maps taking values in some metric space $A$, 
while $\mA$  denotes the set of admissible controls.  
We assume that the {\em drift}  
$
b\colon \T^d \times A \to \R^m \text{ is Lipschitz continuous, locally in } a \in A,
$
then the {\em cost functional} is given by
\bel{cost}
J(\a, m) =   \liminf_{T\to \infty} \frac1T \bE\left[ \int_0^T\big(L(\xi_t, \a_t) 
 +V[m](\xi_t) \big) \, dt \right],
\eeq
for all $\a \in \mA$, where $0\le t \mapsto x(t) \in \T^d$ is the solution to \eqref{dyn}
corresponding to the control $\alpha$. Under our assumptions this solution is uniquely determined by $\a$ and the initial 
condition $x_0\in \T^d$. We will omit to write  explicitly the dependence  on $x_0$ in the functional $J$ since 
the {\em optimal value} is indeed independent of $x_0$.
We assume that  the Lagrangian 
$
L \colon \T^d \times \R^m \to \R $
is measurable and locally bounded and that
$
V \colon P(\T^d) \to L^\infty(\T^d)
$ is measurable.
The standard  MFG theory leads to  the following mathematical problem:\par
 $(\bf P)$:
 Find a pair $(m,   \hat \alpha) \in P(\T^d) \times \mA$ such that 
\begin{enumerate}
\item $\hat\a=\hat\a(m)$ minimizes $J(\cdot, m)$ among $\a \in \mA$, 
\item $m$ is the ergodic measure of the {\em optimal dynamic} $\hat x(\cdot)$ corresponding to the {\em optimal control} $\hat \a$, i.e.
the solution to \eqref{dyn} for $\a=\hat\a$ with   initial state $x_0$.  
\end{enumerate}
The  Hamiltonian has the standard structure
\begin{equation}
 \label{H-def}
H(x, q ) = \sup_{a \in A} \big( -b(x, a) \cdot q -L(x, a) \big) \qquad \forall x \in \T^d\, ,\; q \in \R^m\,.
\end{equation}
and 
we assume that
there exists $ \alpha \colon \T^d \times \R^m \to A$ Lipschitz continuous, locally in  $q \in \R^m $
and such that
$\forall x\in\T^d, q\in \R^m$
the function$ A\ni a \mapsto  -b(x,a)  \cdot q -L(x, a) \in \R$
 attains a maximum at 
$\bar\alpha (x,q)$.
Finally  the auxiliary map 
$
g\colon \T^d \times \R^m \to \R^m  
$ is 
defined by 
$
g(x, q) =  b\big( x,  \bar \alpha (x, q)    \big), $ for all $x\in \T^d \,, q\in \R^m
$.
Note that by the above assumptions  $g$ is Lipschitz, locally in $q\in \R^m$. 


\begin{lem}[Verification theorem]\label{ver-P}
Under all previous assumptions, let
 $(\lambda, u, m) \in \R\times C_{\mX}^2(\T^d) \times P(\T^d)$ be a solution to \eqref{sys-erg}
and  $\hat \alpha$ be the admissible control corresponding to the {\em feedback control} $\bar \alpha\big(x, D_{\mX} u(x)  \big)$, 
that is $\hat \alpha_t = \bar \alpha\big(\hat \xi , D_{\mX} u(\hat\xi_t)  \big)$ for all $t\in [0,+\infty)$, where
$\hat \xi$ is the solution to \eqref{dyn}  for  $\alpha_t= \hat \alpha_t$ and for some $x_0\in \T^d$,
then the pair 
$
\big( \hat \alpha, m   \big)
$
is a solution to problem $(P)$. Moreover
$
\lambda= J(\hat \alpha, m).
$
\end{lem}
\begin{proof}

The proof is trivial but we briefly sketch the main steps for sake of completeness. 
As consequence of It\^o formula and the first equation in
\eqref{sys-erg}, the first property in problem $(\bf P)$ is satisfied. 
Then every Markov process with a compact state space has an invariant measure (e.g. see~\cite[Theorem~9.3]{EK86}). In particular, the diffusion $\hat\xi$ 
has an invariant measure. 
This invariant (actually, ergodic) measure is a weak solution of the dual operator of the generator of $x \mapsto \hat\xi^x$, that means of the dual of $-\mL  + b\big(x, \bar \alpha(x, D_{\mX }u(x) ) \big) \cdot  D_{\mX}= -\mL -g(x, D_{\mX}u ) \cdot D_{\mX}$. In other words, this invariant measure is a solution of equation~\eqref{ad-pr-eq}. But by Lemma~\ref{adj-prob} that equation has a unique solution, so the invariant measure of $\hat\xi$ is precisely $m$.
Then also  condition 2  is satisfied. 
\end{proof}
Combining together Theorem~\ref{erg-thm-lin} and Lemma~\ref{ver-P} one can derive the following result.
\begin{cor}[Existence of solutions to the MFG problem]
Under the assumptions of  Lemma~\ref{ver-P}
and assuming in addition  that $L$ is locally H\"older continuous, then
for every $x_0 \in \T^d$ there exists a solution $(\hat \alpha, m)$ to problem $(\bf P)$. 
\end{cor}

\subsection{Nash equilibria for a class of $N$-player games}


The solvability 
theory for systems similar to \eqref{sys-erg} can be applied to build Nash equilibria in feedback form for a class of
stochastic differential $N$-player games.   This is a straightforward adaptation to the case of H\"ormander diffusions of the results contained in 
\cite{BF16}. We include the main ideas for completeness. 
Let $N\in \N$, $W_t^i$ be a $\R^{m_i}$-valued Brownian motion, for every $i=1, \ldots, N$ and
for some $m_i \in \N$, adapted to the
filtered probability space $(\Omega, \mF, (\mF)_{t\in \R_+}, \bP)$
and assume that $W_t^1, \ldots, W_t^N$ are independent, then
 the dynamic of the game is described by the system of SDEs 
\bel{dyn-g}
\left\{
\begin{aligned}
d \xi^i_t &=  \sum_{k=1}^{m_i} b_i^k(\xi^i_t,  \a^i_t )  X^i_k(\xi^i_t) dt + \sum_{k=1}^{m_i}X^i_k(\xi^i_t) \circ dW_t^i   
  \quad  \text{in } \T^{d_i} \\
\xi^i_0&= x_0^i 
\end{aligned} \right. \, , \qquad i= 1, \ldots, N,
\eeq
where 
$\xi^i$ is the state of the $i$-th player, 
$x_0^i$ are given initial conditions,  
 $A_i$ is a given metric space
 and 
 the {\em set of control parameters of player $i$} and each admissible control (namely also strategy) of player $i$,
$\a^i\colon: \R_{+} \times \Omega \to A_i  $, is a measurable and locally bounded map
 adapted to $W_t^i$.
Let $\mA_i$ denotes the set of all
admissible controls for player $i$, assume that
$
\mX_i= \{X^i_1, \ldots, X^i_{m_i}\}
$
 is a set of smooth H\"ormander vector fields
on the flat torus $ \T^{d_i} $ for some  $d_i \in \N $, 
and the drift
$
b_i\colon \T^{d_i} \times A_i \to \R^{m_i}$  is a locally Lipschitz map. 
Under these assumptions, it is known that for any $N$-tuple of initial conditions $(x_0^1, \ldots, x_0^N)$
and for every $N$-tuple of admissible controls $(\alpha^1, \ldots, \alpha^N) \in \mA_1 \times \cdots \times \mA_N$ there exists 
a unique solution $\xi = (\xi^1, \ldots, \xi^N)$ to  \eqref{dyn-g}. Actually, the system \eqref{dyn-g} is ``decoupled'' in
the sense that each SDE for $\xi^i$ 
is solved independently of all the other equations 
and  the stochastic processes $\xi^1, \ldots, \xi^N$ are independent of each other: each $\xi^i$ is adapted to its ``own'' Brownian motion
$W_t^i$. 
The {\em cost} (or {\em performance criterion}) of player  $i$ 
is given by 
\bel{cost-i}
J_i(\a^1, \ldots, \a^N) = \lim_{T\to \infty} \frac1T \bE \left[\int_0^T \big( L_i(\xi^i_t, \a^i_t) + F_i(\xi^1_t, \ldots, \xi^N_t)   \big)\, dt  \right] \,,
\eeq
where $
L_i\colon \T^{d_i}\times A_i \to \R $
and
$
F_i: \T^{d_1} \times \cdots \times \T^{d_N} \to \R$
are H\"older continuous.
Each player seeks to optimise its performance criterion (minimising the cost) in presence of all the competitors. 
Clearly, the agents have conflicting goals: a {\em win-win} set of strategies that satisfies all players,  i.e  minimises all their costs simultaneously, in general will not exist. In these types of problems, a good notion of solution turns out to be the notion of Nash equilibrium:
 a set of admissible strategies $(\hat \alpha^1, \ldots, \hat \alpha^N) \in \mA_1\times\cdots \times \mA_N$ 
 is called a {\em Nash equilibrium}
 if, for every $i=1, \ldots N$, 
$
J_i( \hat \alpha^1, \ldots, \hat \alpha^N )  \le J_i(\hat \alpha^1, \ldots,  \hat\alpha^{i-1}, \alpha^i , \hat\alpha^{i+1} \ldots \hat \alpha^N)
$.
In other words, the player $i$ cannot ``perform better'' by moving away from $\hat \alpha^i$ if the opponents continue to stick to 
$(\hat \alpha^1, \ldots, \hat \alpha^N )$.
The problem of finding Nash equilibria  
 reduces  to finding solutions to a system of $2N$ PDEs, made by $N$ equations of HJB type coupled with $N$ equations 
of KFP type:
\bel{sys-N}
\left\{\begin{aligned}
& \mL_i u_i + \l_i + H_i(x, D_{\mX_i} u_i ) = V^i[m_1, \ldots, m_N] \quad  \text{in } \T^{d_i},  \\ 
& \mL_i^*m_i + \textrm{div}_{\mX_i^*} \big( m_i g_i\big( x_i, D_{\mX_i}u_i  \big ) \big) =0 \quad \text{in } \T^{d_i} ,\\
& \int_{\T^{d_i}}u_i \, dx^i=0 \, , \quad \int_{\T^{d_i}}m_i \, dx^i=0 , \quad m_i >0,
 \end{aligned}\right. \quad i=1, \ldots, N \,,
\eeq
where $\lambda_i\in \R$, $H_i(x, q ) = \sup_{a \in A_i} \big( -b_i(x, a) \cdot q - L_i(x, q) \big)$, $\mL_i = - \frac12 \sum_{k=1}^{m_i} (X^i_k)^2 $ (with dual operator $\mL_i ^*$), the auxiliary maps
$
g_i\colon \T^{d_i} \times \R^{m_i} \to \R^{m_i} 
$ are defined as 
$ g_i(x,q) = b_i \big(x, \bar \alpha(x, q) \big)$ $\forall
x \in \T^{d_i}, \; q \in \R^{m_i}$
and the operators 
$
V_i\colon \prod_{\substack{ 1\le j\le N \\ j\neq i}}   P(\T^{d_i})  \to L^\infty(\T^{d_i})
$
are defined by 
\bel{Vi-def}
V_i[m_1, \ldots, m_{i-1}, m_{i+1}, \ldots, m_N ](x) = \int_{ \prod_{\substack{ j=1 \\ j\neq i}}^N  \T^{d_i}    } F(x_1, \ldots, x_N) \prod_{\substack{ 1\le j \le N \\ j \neq i  } } \, m_j(dx^j)\,. 
\eeq

\begin{thm}[PDEs and Nash equilibria]\label{Nash}
Assume 
that
there exist maps 
$
\bar \alpha^i \colon \T^{d_i} \times \R^{m_i}    \to A_i,$
Lipschitz continuous,
such that 
$
\forall x\in \T^{d_i},  q \in \R^{m_i}$ $\bar\alpha^i(x,q) $ is a maximum point of 
$A_i \ni a \mapsto  -b_i(x, a) \cdot q - L_i(x, \alpha) \in \R$,
and  the Hamiltonians $H_i$
grow at most linearly in the gradient variable, uniformly w.r.t. $x$,
then 
 \begin{enumerate}[$(i)$]
\item
There exists a solution $\l_i\in \R$, $u_i \in C^2_{\mX_i}(\T^d)$, $m_i \in H^1_{\mX_i}(\T^{d_i}) \cap L^\infty(\T^{d_i})$, $i=1, \ldots, N$
to system \eqref{sys-N}. 
\item Every solution of \eqref{sys-N} determines a Nash equilibrium in {\em feedback form}
by 
$
\hat \a^i(x) = \bar\alpha^i\big(x, D_{\mX_i}u_i(x)   \big),
$
for the game described above. 
Moreover $\l_i = J_i(\hat \alpha^1, \ldots, \hat\alpha^N )$
and $m_i$ is the ergodic measure of  $\hat \xi^i$, where  $\hat \xi 
=(\hat \xi^1, \ldots, \xi^N)$ is the  optimal dynamic \eqref{dyn}
corresponding to the optimal control $\hat\a =(\hat \alpha^1, \ldots, \hat\alpha^N)$.
\end{enumerate}
\end{thm} 
Statement (i) is a corollary of Theorem~\ref{erg-thm-lin}. 
Statement (ii) is the analog of a so-called verification theorem in optimal control and differential games. As hinted in \cite{LL-jj} the proof is standard and relies here also on an ergodic theorem for 
H\"ormander diffusions. See \cite{BF16} for more details in the case of uniformly elliptic diffusions.

\subsection{Mean Field Games as limit of $N$-player games}
Considering the same game  as in the previous subsection,
the goal  is to let the number of the players  $N\to \infty$ .  In order to be able to pass to the limit as $N\to \infty$ we have to make two additional assumptions.
\begin{enumerate}
\item The players are {\em similar}. Mathematically this means that we are assuming that  
all $d_i$, $\mX_i$, $H_i$, $L_i$, $F_i$ are the same, that is, independent of $i$. 
Being similar, the agents will reason ``similarly'', so we can assume also that
$
\bar\alpha^i = \bar\alpha \quad \forall i=1, \ldots, N. 
$
As a consequence all $H_i$ and all $g_i$ will be the same, and we can call them $H$ and $g$. 
\item Since players are ``small'' and their number is ``large'', each player can only have a ``statistical visibility'' of the game, 
each player cannot know all the individual states  of the agents taking part in the game, but he knows, for example, the average or their states 
(some ``macroeconomic parameter'' say, that can somehow be measured or estimated).  
Mathematically this can be expressed assuming:
\[
F(x^1, \ldots, x^N) = W\left[ \frac1N \sum_{j=1}^N \delta_{x^j} \right](x^i), \qquad \forall (x^1, \ldots x^N) \in (\T^d)^N, 
\]
for some map
$
W\colon P(\T^d) \to L^\infty(\T^d) \,, 
$
which we assume satisfying condition {\bf (IV)},
(and  $\delta_x$ denotes the usual Dirac delta measure). 
Thus we are assuming that each agent   designs his cost 
 as a function of the {\em empirical average} $\frac1N \sum_{j=1}^N \delta_{x^j} $ of the states of all the agents  and  of its own state $x^i$ (after all, it is reasonable to expect that each knows at least his own state $x^i$ and treats it separately from the states of the rest of the agents)
\end{enumerate}
Under these assumptions system \eqref{sys-N} reduces to 
\bel{sys-N-s}\left\{\begin{aligned}&\mL u_i + \l_i + H(x, D_{\mX} u_i ) = V[m_1, \ldots, m_N] \quad\text{in } \T^d,
   \\ & \mL^*m_i - \textrm{div}_{\mX^*} \big( m_i g \big( x_i, D_{\mX}u_i  \big ) \big) =0 \quad \text{in } \T^{d} \qquad i=1, \ldots, N\,,\\ 
	&\int_{\T^{d}}u_i \, dx=0 \, , \quad \int_{\T^d}m_i \, dx=0 , \quad m_i >0,  \end{aligned}\right. 
\eeq 
with 
\bel{V-def}
V[m_1, \ldots, m_N](x^i) =
  \int_{   \T^{d(N-1)}  } W\left[ \frac1N \sum_{j=1}^N \delta_{x^j}   \right](x^i) \prod_{\substack{ 1\le j \le N \\ j \neq i  } } \, m_j(dx^j)
\quad \forall x^i \in \T^d\,. 
\eeq
\begin{rem}[Existence of symmetric solutions]
Under the above assumptions it is easy to adapt the proof of 
Theorem~\ref{Nash}~(i) or of Theorem~\ref{erg-thm-lin} in order to show 
that the system of PDEs  \eqref{sys-N-s} has a symmetric solution
$(\l, \ldots, \l) \in\R^N$, $(u, \ldots, u) \in C^2_{\mX}(\T^d)$ and 
$(m, \ldots, m) \in \big( H^1_{\mX} \cap L^\infty(\T^d) \big)^N$. 
\end{rem}
System~ \eqref{sys-N-s} does not have a unique solution in general. However, all its solutions symmetrise as $N\to \infty$, (as shown in the following theorem). 
Moreover, the limit points  
of these solutions satisfy the system of MFG equation. 
We  recall that $P(\T^d) \subset C(\T^d)^*$ is a compact topological space for the topology of weak $^*$-convergence (Prokhorov's Theorem). 
Moreover, this topology is metrizable, for example, by the {\em Kantorovich-Rubinstein} distance $\bd(m_1, m_2)$ defined for $m_1, m_2 \in P(\T^d)$ as
$
\bd(m_1, m_2) = \sup\left\{  \int_{\T^d} f(x) \, d(m_1-m_2)  \; : \; f\in C^{0,1}, \; {\rm Lip}\,(f) \le 1 \right\}.
$

\begin{thm}[Symmetrisation and MFG limit]
Let $(\lambda_1^N, \ldots, \lambda_N^N) \in \R^N$, $(u_1^N, \ldots, u_N^N) \in C^2_{\mX}(\T^d)$,
$(m_1^N, \ldots, m_N^N)\in \big(H^1_{\mX}(\T^d) \cap L^\infty(\T^d)\big)^N$ be a solution to  \eqref{sys-N} -
\eqref{V-def}.  Then 
 \begin{enumerate}[$(i)$]
\item $\{ (\l_i^N, u_i^N, m_i^N)\}_{N\ge i}$ is precompact in $\R\times C^2_{\mX}(\T^d) \times P(\T^d)$
for every $i\in \N$. 
\item (symmetrisation:)
$
\lim_{N \to \infty} \left(  |\lambda_i^N- \lambda_j^N| +\|u_i^N-u_j^N\|_{C_{\mX}^2(\T^d)} + \,d(m_i^N, m_j^N) \right) =0\,. 
$
\item Let $(\l, u, m)$ be a limit point of $\{ (\l_i^N, u_i^N, m_i^N)\}_{N\ge i}$ for some $i\in \N$. Then $(\l, u, m)$
is a solution of the MFG system
\bel{sys-W}
\left\{\begin{aligned} 
&\mL u + \lambda + H(x, D_{\mX}u ) = W[m] \text{ in } \T^d \\
&\mL^*u - \textrm{div}_{\mX^*} \big( m g(x, D_{\mX}u  )\hat \alpha ) \big)=0 \text{ in }\T^d \\
&\int_{\T^d}u \, dx=0, \quad \int_{\T^d}m \,dx=1, \quad m >0\,. 
\end{aligned} \right.
\eeq 
\end{enumerate} 
\end{thm}
\begin{proof}
(i) is a consequence of the a priori estimates for solutions of system \eqref{sys-N-s}, which one can easily show that under the current assumptions hold true with constants independent of $N$. 
 As hinted in~\cite{LL-cs}, the proof of (ii) relies on  the uniqueness and continuous dependence of solutions of HJB and KFP equations on the data, while for the proof of (iii) one needs also a law or large numbers. 
For additional details 
we refer to \cite{Fe13}. 
\end{proof}

 \begin{rem}
\begin{enumerate}
\item[(i)]
If the MFG system has a unique solution $(\l, u, m)$  (see Theorem~\ref{uniq-thm} for sufficient conditions), then clearly
$(\l_i^N, u_i^N, m_i^n) \to (\l, u, m)$, as $N\to \infty$, for every $i\in \N$. 
\item[(ii)]
Solving the system of PDEs \eqref{sys-W}, for every $\e>0$, we can build symmetric $\e$-Nash equilibria  
for the $N$-player game,  provided $N$ is sufficiently large. 
\end{enumerate} 
\end{rem}

\bibliographystyle{plain}
\def\cprime{$'$}

\end{document}